\documentclass[12pt,english]{amsart}

\usepackage[right=2.5cm, left=2.5cm, top=2.5cm, bottom=2.5cm]{geometry}
\usepackage[all]{xy}
\usepackage{amsmath}
\usepackage{amssymb}
\usepackage{appendix}
\usepackage{hyperref}
\newcommand{\Qq}{\mathbb{Q}}

\numberwithin{equation}{section}

{\theoremstyle{plain}
\newtheorem{theorem}{Theorem}[section]
\newtheorem{lemma}[theorem]{Lemma}
\newtheorem{proposition}[theorem]{Proposition}
\newtheorem{conjecture}[theorem]{Conjecture}
\newtheorem{corollary}[theorem]{Corollary}}

{\theoremstyle{remark}
\newtheorem{remark}[theorem]{Remark}}

\title{Embedding problems for automorphism groups of field extensions}

\author{Arno Fehm}
\email{arno.fehm@tu-dresden.de}

\author{Fran\c cois Legrand}
\email{francois.legrand@tu-dresden.de}

\author{Elad Paran}
\email{paran@openu.ac.il}

\address{Institut f\"ur Algebra, Fachrichtung Mathematik, TU Dresden, 01062 Dresden, Germany}

\address{Institut f\"ur Algebra, Fachrichtung Mathematik, TU Dresden, 01062 Dresden, Germany}

\address{Department of Mathematics and Computer Science, the Open University of Israel, Ra'anana 4353701, Israel}

\date{\today}



\begin{document}

\def\commutatif{\ar@{}[lldd]|{\circlearrowleft}}
\def\commutative{\ar@{}[rrdd]|{\circlearrowleft}}

\maketitle

\begin{abstract}
A central conjecture in inverse Galois theory, proposed by D\`ebes and Des-champs, asserts that every finite split embedding problem over an arbitrary field can be regularly solved. We give an unconditional proof of a consequence of this conjecture, namely that such embedding problems can be regularly solved if one waives the requirement that the solution fields are normal. This extends previous results of M. Fried, Takahashi, Deschamps, and the last two authors concerning the realization of finite groups as automorphism groups of field extensions.
\end{abstract}

\section{Introduction} \label{sec:intro}

Understanding the structure of the absolute Galois group ${\rm{G}}_{\Qq}$ of $\mathbb{Q}$ is one of the central objectives in algebraic number theory, and has inspired a number of very different approaches. One of them, which shall be our focus, is classical inverse Galois theory in the tradition of Hilbert and E.\ Noether. The first question here, which was studied already in the late 19th century and is yet still open, is the inverse Galois problem, which asks whether all finite groups occur as quotients of ${\rm{G}}_{\Qq}$, i.e.,~as Galois groups of Galois extensions of $\mathbb{Q}$. More information on ${\rm{G}}_{\Qq}$ would be obtained by knowing which finite {\em embedding problems} over $\Qq$ are solvable, or which finite groups have geometric, so-called {\em regular} realizations over $\Qq$ (see \cite{Vol96, MM99, FJ08} or below for more details). 

The central conjecture in this area, which suggests answers to all of these questions, consistent with what is known and expected about each of them, and without needing to restrict to the field $\mathbb{Q}$, was formulated by D\`ebes and Deschamps (see \cite[\S2.2]{DD97b}):

\begin{conjecture} \label{conj:DD1}
Let $k$ be a field, $G$ a finite group, $L$ a finite Galois extension of the rational function field $k(T)$, and
$\alpha : G \rightarrow {\rm{Aut}}(L/k(T))$ an epimorphism that has a section\footnote{i.e., there exists an embedding $\alpha' : {\rm{Aut}}(L/k(T)) \rightarrow G$ such that $\alpha \circ \alpha' = {\rm{id}}_{{\rm{Aut}}(L/k(T))}$.}. Then there are a finite Galois extension $E/k(T)$ with $L \subseteq E$ and $E \cap \overline{k} = L \cap \overline{k}$, and an isomorphism $\beta : {\rm{Aut}}(E/k(T)) \rightarrow G$ such that $\alpha \circ \beta$ is the restriction map ${\rm{Aut}}(E/k(T)) \rightarrow {\rm{Aut}}(L/k(T))$.
\end{conjecture}

\noindent
In particular, Conjecture \ref{conj:DD1} would imply that every finite group occurs as the Galois group of a Galois extension of $\mathbb{Q}$, or, more generally, of every field that is {\em Hilbertian}, i.e., for which Hilbert's irreducibility theorem holds\footnote{For example, all global fields are Hilbertian. See, e.g., \cite{FJ08} for more on Hilbertian fields.}. So far, Conjecture \ref{conj:DD1} has been proved only if $k$ is an ample field\footnote{Recall that a field $k$ is {\it{ample}} (or {\it{large}}) if every smooth $k$-curve has either zero or infinitely many $k$-rational points. See, e.g., \cite{Jar11} and \cite{BSF13} for more about ample fields.} (see \cite[Main Theorem A]{Pop96} and \cite[Theorem 2]{HJ98b}), and no counterexample is known.

However, certain consequences of Conjecture \ref{conj:DD1}, like the inverse Galois problem, were proven in a weak form. Starting with \cite{FK78}, the easier question whether every finite group occurs as the automorphism group of a finite extension of $\mathbb{Q}$, not necessarily Galois, was studied. Clearly, a positive answer to this question is necessary for a positive solution to the inverse Galois problem, hence the interest in the question. The work \cite{LP18}, which extends previous results of M. Fried \cite{Fri80} and Takahashi \cite{Tak80} on this question, shows that indeed all finite groups occur as automorphism groups of finite extensions of any Hilbertian field. In \cite{DL18}, this was strengthened to `regular' realizations (for an arbitrary field); see th\'eor\`eme 1 of that paper for more details.

In light of these results, it is natural to ask whether in fact Conjecture \ref{conj:DD1} holds unconditionally if one waives the requirement that the extension $E/k(T)$ is normal. Once again, an affirmative answer to this question is necessary for an affirmative answer to the conjecture. In the present work, we show that this indeed is the case, thereby also generalizing the results mentioned above:

\begin{theorem} \label{thm:intro}
The statement of Conjecture \ref{conj:DD1} holds unconditionally if we do not require the extension $E/k(T)$ to be normal.
\end{theorem}

\begin{corollary} \label{coro:intro}
Let $k$ be a Hilbertian field, $L/k$ a finite Galois extension, $G$ a finite group, and $\alpha:G \rightarrow {\rm{Aut}}(L/k)$ an epimorphism that has a section. Then there exist a finite separable extension $F/L$ and an isomorphism $\beta : {\rm{Aut}}(F/k) \rightarrow G$ such that $\alpha \circ \beta$ is the restriction map ${\rm{Aut}}(F/k) \rightarrow {\rm{Aut}}(L/k)$.
\end{corollary}

\noindent
Note that Corollary \ref{coro:intro} is new already for $k=\mathbb{Q}$. An even more general, yet more technical version of Theorem \ref{thm:intro} is given in Theorem \ref{thm:main}, where the extension $L/k(T)$ is not necessarily Galois and the epimorphism $\alpha$ does not necessarily have a section, and accordingly Corollary \ref{coro:intro} holds in this greater generality (see Corollary \ref{coro:spec}).

\vspace{3mm}

{\bf{Acknowledgements.}} We wish to thank Lior Bary-Soroker for suggesting to use \cite[Proposition 16.11.1]{FJ08} in the proof of Proposition \ref{lemma:0}, and Dan Haran for his help with Proposition \ref{prop:pac}. 
%This work is partially funded by ... (@Elad: if you need to quote some grant numbers from the ISF, please add them here; there is nothing to quote regarding TU Dresden).

\section{Terminology and notation} \label{sec:basics}

The aim of this section is to state some terminology and notation on restriction maps, finite embedding problems, and specializations of function field extensions in our non-Galois context. For this section, let $k$ be an arbitrary field and $\overline{k}$ an algebraic closure of $k$.

\subsection{Restriction maps} \label{ssec:rm}

Let $L/k$ be a finite separable extension. Recall that ${\rm{Aut}}(L/k)$ is the automorphism group of $L/k$, i.e., the group of all isomorphisms of the field $L$ which fix $k$ pointwise. If $L/k$ is Galois, then this group is the Galois group of $L/k$, denoted by ${\rm{Gal}}(L/k)$.

Let $L/k$ and $F/M$ be two finite separable extensions with $k \subseteq M$ and $L \subseteq F$. Let $H$ be the subgroup of ${\rm{Aut}}(F/M)$ consisting of all elements fixing $L$ setwise. The restriction map
$$\left \{ \begin{array} {ccc}
H & \longrightarrow & {\rm{Aut}}(L/k) \\
\sigma & \longmapsto & \sigma{|_L}
\end{array} \right. $$
shall be denoted by ${\rm{res}}_{L/k}^{F/M}$. Note that the map ${\rm{res}}_{L/k}^{F/M}$ is not necessarily surjective. Moreover, if $N/K$ is a finite separable extension such that $M \subseteq K$ and $F \subseteq N$, then the composed map
$${\rm{res}}_{L/k}^{F/M} \circ {\rm{res}}^{N/K}_{F/M}$$
is not defined in general, and if it is, it is not necessarily equal to ${\rm{res}}^{N/K}_{L/k}$. However, these two properties hold if the domains of ${\rm{res}}_{L/k}^{F/M}$ and ${\rm{res}}^{N/K}_{F/M}$ are ${\rm{Aut}}(F/M)$ and ${\rm{Aut}}(N/K)$, respectively (in particular, if $L/k$ and $F/M$ are Galois).

Let $L/k$ be a finite separable extension and $\widehat{L}$ the Galois closure of $L$ over $k$. If $H$ is the normalizer of ${\rm{Gal}}(\widehat{L}/L)$ in ${\rm{Gal}}(\widehat{L}/k)$, then ${\rm{res}}_{L/k}^{\widehat{L}/k}$ has domain $H$, it is surjective, and it induces an isomorphism between the quotient group $H/{\rm{Gal}}(\widehat{L}/L)$ and the group ${\rm{Aut}}(L/k)$. Hence, if $M$ is any field containing $k$ that is linearly disjoint from $\widehat{L}$ over $k$, then the domain of ${\rm{res}}_{L/k}^{LM/M}$ is the whole automorphism group ${\rm{Aut}}(LM/M)$ and this map is an isomorphism from ${\rm{Aut}}(LM/M)$ to ${\rm{Aut}}(L/k)$ \footnote{If we only assume that the fields $L$ and $M$ are linearly disjoint over $k$, then the groups ${\rm{Aut}}(LM/M)$ and ${\rm{Aut}}(L/k)$ are not isomorphic in general. For example, $\Qq(\sqrt[3]{2})$ and $\Qq(e^{2i\pi/3})$ are linearly disjoint over $\Qq$ (as they have coprime degrees). But ${\rm{Aut}}(\Qq(\sqrt[3]{2}) / \Qq)$ is trivial while ${\rm{Gal}}(\Qq(\sqrt[3]{2},e^{2i\pi/3}) / \Qq(e^{2i\pi/3}))$ has order 3.}. In particular, if ${\bf{T}}$ is a finite tuple of algebraically independent indeterminates, then this holds for $M=k({\bf{T}})$, that is, the restriction map ${{\rm{res}}_{L/k}^{L({\bf{T}})/k({\bf{T}})}} : {\rm{Aut}}(L({\bf{T}})/k({\bf{T}})) \rightarrow {\rm{Aut}}(L/k)$ is an isomorphism.

\subsection{Finite embedding problems} \label{ssec:fep}

The terminology below extends standard terminology like in \cite[Definition 16.4.1]{FJ08}.

A {\it{finite embedding problem over $k$}} is an epimorphism $\alpha : G \rightarrow {\rm{Aut}}(L/k)$, where $G$ is a finite group and $L/k$ a finite separable extension. Say that $\alpha$ is {\it{split}} if there is an embedding $\alpha' : {\rm{Aut}}(L/k) \rightarrow G$ such that $\alpha \circ \alpha' = {\rm{id}}_{{\rm{Aut}}(L/k)}$, and that $\alpha$ is {\it{Galois}} if $L/k$ is Galois. A {\it{solution}} to $\alpha$ is an isomorphism $\beta : {\rm{Aut}}(E/k) \rightarrow G$, where $E/k$ is a finite separable extension such that $L \subseteq E$, that satisfies $\alpha \circ \beta = {\rm{res}}_{L/k}^{E/k}$ \footnote{This implies in particular that the domain of the restriction map ${\rm{res}}_{L/k}^{E/k}$ is the whole group ${\rm{Aut}}(E/k)$ and that this map is surjective.}. Refer to $E$ as the {\it{solution field}} associated with $\beta$. In the case $\alpha$ is Galois, the solution $\beta$ is {\it{Galois}} if the extension $E/k$ is Galois.

Let $\alpha: G \rightarrow {\rm{Aut}}(L/k)$ be a finite embedding problem over $k$ and let $\widehat{L}$ be the Galois closure of $L$ over $k$. If $M$ denotes any field containing $k$ which is linearly disjoint from $\widehat{L}$ over $k$, then the finite embedding problem
$${{\rm{res}}_{L/k}^{LM/M}}^{-1} \circ \alpha : G \rightarrow {\rm{Aut}}(LM/M)$$
over $M$ is denoted by $\alpha_{M}$ \footnote{As seen in \S\ref{ssec:rm}, the restriction map ${\rm{res}}_{L/k}^{LM/M} : {\rm{Aut}}(LM/M) \rightarrow {\rm{Aut}}(L/k)$ is a well-defined isomorphism.}. If $\beta$ is a solution to $\alpha_M$ with solution field denoted by $E$, then
$${\rm{res}}^{E/M}_{LM/M} = \alpha_M \circ \beta= {{\rm{res}}_{L/k}^{LM/M}}^{-1} \circ \alpha \circ \beta.$$
In particular, one has
$\alpha \circ \beta= {\rm{res}}_{L/k}^{LM/M} \circ {\rm{res}}^{E/M}_{LM/M} = {\rm{res}}^{E/M}_{L/k}.$

Given an indeterminate $T$, let $\alpha: G \rightarrow {\rm{Aut}}(L/k(T))$ be a finite embedding problem over $k(T)$. A solution to $\alpha$ is {\it{regular}} if the associated solution field $E$ satisfies $E \cap \overline{k} = L \cap \overline{k}$.

Let $\alpha : G \rightarrow {\rm{Aut}}(L/k)$ be a finite embedding problem over $k$. A {\it{geometric solution}} to $\alpha$ is a regular solution $\beta$ to $\alpha_{k(T)}$. Furthermore, in the case $L/k$ is Galois, we shall say that $\beta$ is a {\it{geometric Galois solution}} to $\alpha$ if $\beta$ is a regular Galois solution to $\alpha_{k(T)}$.

\subsection{Specializations} \label{ssec:ffe}

For more on the following, we refer to \cite[\S1.9]{Deb09} and \cite[\S2.1.4]{DL13}. Let ${\bf{T}}=(T_1, \dots, T_n)$ be an $n$-tuple of algebraically independent indeterminates ($n \geq 1$), $E/k({\bf{T}})$ a finite separable extension, and $\widehat{E}$ the Galois closure of $E$ over $k({\bf{T}})$.

Let $\widehat{B}$ be the integral closure of $k[{\bf{T}}]$ in $\widehat{E}$. For ${\bf{t}}=(t_1, \dots, t_n) \in k^n$, the residue field of $\widehat{B}$ at a maximal ideal $\mathcal{P}$ lying over the ideal $\langle {\bf{T}} - {\bf{t}} \rangle$ of $k[{\bf{T}}]$ generated by $T_1-t_1, \dots, T_n-t_n$ is denoted by $\widehat{E}_{\bf{t}}$ and the extension $\widehat{E}_{\bf{t}}/k$ is called the {\it{specialization}} of $\widehat{E}/k({\bf{T}})$ at ${\bf{t}}$. As $\widehat{E}/k({\bf{T}})$ is Galois, the field $\widehat{E}_{\bf{t}}$ does not depend on $\mathcal{P}$ and the extension $\widehat{E}_{\bf{t}}/k$ is finite and normal. Moreover, for ${\bf{t}}$ outside a Zariski-closed proper subset (depending only on $\widehat{E}/k({\bf{T}})$), the extension $\widehat{E}_{\bf{t}}/k$ is Galois and its Galois group is a subgroup of ${\rm{Gal}}(\widehat{E}/k({\bf{T}}))$, namely the decomposition group of $\widehat{E}/k({\bf{T}})$ at $\mathcal{P}$. If $P({\bf{T}},X) \in k[{\bf{T}}][X]$ is the minimal polynomial of a primitive element of $\widehat{E}$ over $k({\bf{T}})$, assumed to be integral over $k[{\bf{T}}]$, then, for ${\bf{t}} \in k^n$, the field $\widehat{E}_{\bf{t}}$ contains a root $x_{\bf{t}}$ of $P({\bf{t}},X)$. In particular, if $P({\bf{t}},X)$ is irreducible over $k$ and separable, then $\widehat{E}_{\bf{t}} = k(x_{\bf{t}})$ and $[\widehat{E}_{\bf{t}} : k]=[\widehat{E}:k({\bf{T}})]$, the extension $\widehat{E}_{\bf{t}}/k$ is Galois, and $\widehat{E}_{\bf{t}}$ is the splitting field over $k$ of $P({\bf{t}},X)$.

Let $B$ be the integral closure of $k[{\bf{T}}]$ in ${E}$. For ${\bf{t}} \in k^n$, let $\mathcal{P}_1, \dots, \mathcal{P}_s$ be the maximal ideals of $B$ lying over $\langle {\bf{T}} - {\bf{t}} \rangle$. For $i \in \{1, \dots,s\}$, the residue field of $B$ at $\mathcal{P}_i$ is denoted by $E_{{\bf{t}},i}$ and the finite extension ${E}_{{\bf{t}},i}/k$ is called a {\it{specialization}} of ${E}/k({\bf{T}})$ at ${\bf{t}}$. If $\widehat{E}_{\bf{t}}/k$ is Galois, then ${E}_{{\bf{t}},1}/k, \dots, {E}_{{\bf{t}},s}/k$ are separable. Moreover, if $[\widehat{E}_{\bf{t}} : k]=[\widehat{E}:k({\bf{T}})]$, then $s=1$, in which case the field $E_{{\bf{t}},1}$ is simply denoted by $E_{{\bf{t}}}$, one has $[{E}_{\bf{t}} : k]=[{E}:k({\bf{T}})]$, and if ${E}_{{\bf{t}}}/k$ is separable, then $\widehat{E}_{\bf{t}}$ is the Galois closure of $E_{\bf{t}}$ over $k$.

\section{Two general results about finite embedding problems} \label{sec:series}

This section is devoted to two general results about finite embedding problems (Galois or not) that will be used in \S\ref{sec:proof} to prove Theorem \ref{thm:main} and Corollary \ref{coro:spec}; see Propositions \ref{lemma:0} and \ref{prop:sssecI}. For this section, let $k$ be an arbitrary field.

\subsection{Specializing indeterminates} \label{ssec:spe_ind}

Let ${\bf{T}}=(T_1, \dots, T_n)$ be a tuple of algebraically independent indeterminates ($n \geq 1$).

Given a finite Galois embedding problem $\alpha : G \rightarrow {\rm{Gal}}(L/k)$ over $k$, it is classical that if $k$ is Hilbertian and $\alpha_{k({\bf{T}})}$ has a Galois solution whose solution field is denoted by $E$, then $\alpha$ has a Galois solution whose associated solution field is a suitable specialization of $E$. See, e.g., \cite[Lemma 16.4.2]{FJ08} for more details. We now extend this to our non-Galois context:

\begin{proposition} \label{lemma:0}
Let $\alpha:G \rightarrow {\rm{Aut}}(L/k)$ be a finite embedding problem over $k$. Suppose the finite embedding problem $\alpha_{k({\bf{T}})}$ has a solution, whose solution field is denoted by $E$.

\vspace{0.5mm}

\noindent
{\rm{(1)}} Suppose $k$ is Hilbertian. Then, for each ${\bf{t}}$ in a Zariski-dense subset of $k^n$, the embedding problem $\alpha$ has a solution whose associated solution field is $E_{{\bf{t}}}$.

\vspace{0.5mm}

\noindent
{\rm{(2)}} Suppose $k =\kappa(T)$ for some field $\kappa$ (and $T$ an indeterminate) and $E \cap \overline{\kappa} = L \cap \overline{\kappa}$. Then there exists ${\bf{t}} \in k^n$ such that $E_{\bf{t}}$ is the solution field of a regular solution to $\alpha$.
\end{proposition}

\begin{proof}
We break the proof into three parts. Let $\widehat{E}$ denote the Galois closure of $E$ over $k({\bf{T}})$.

Firstly, let ${\bf{t}} \in k^n$. We claim that if $[\widehat{E}_{\bf{t}} : k]=[\widehat{E}:k({\bf{T}})]$ and $\widehat{E}_{\bf{t}}/k$ is Galois, then the field $E_{{\bf{t}}}$, which is well-defined (see \S\ref{ssec:ffe}), is the solution field of a solution to $\alpha$. From our assumptions on ${\bf{t}}$, there exists an isomorphism
$$\psi_{\bf{t}}:{\rm{Gal}}(\widehat{E}_{\bf{t}}/k) \rightarrow {\rm{Gal}}(\widehat{E}/k({\bf{T}}))$$
such that the following two conditions hold:
\begin{equation} \label{eq1}
\psi_{\bf{t}}({\rm{Gal}}(\widehat{E}_{\bf{t}} / E_{\bf{t}}))={\rm{Gal}}(\widehat{E}/E),
\end{equation}
\begin{equation} \label{eq1.1}
\psi_{\bf{t}}(\sigma)(x) = \sigma(x) \, \, \, {\rm{for}} \, \, \, {\rm{every}} \, \, \, \sigma \in {\rm{Gal}}(\widehat{E}_{\bf{t}}/k) \, \, \, {\rm{and}} \, \, \, {\rm{every}} \, \, \, x \in L.
\end{equation}
See \cite[Lemma 16.1.1]{FJ08} and \cite[\S1.9]{Deb09} for more details. Denote the normalizer of ${\rm{Gal}}(\widehat{E}/E)$ in ${\rm{Gal}}(\widehat{E}/k({\bf{T}}))$ by $H$ and that of ${\rm{Gal}}(\widehat{E}_{\bf{t}}/E_{\bf{t}})$ in ${\rm{Gal}}(\widehat{E}_{\bf{t}}/k)$ by $H_{\bf{t}}$. From \eqref{eq1},
\begin{equation} \label{eq1.5}
\psi_{\bf{t}}(H_{\bf{t}}) = H.
\end{equation}
As the domains of the maps ${\rm{res}}_{E/k({\bf{T}})}^{\widehat{E}/k({\bf{T}})}$, ${\rm{res}}_{L({\bf{T}})/k({\bf{T}})}^{E/k({\bf{T}})}$, and ${\rm{res}}_{L/k}^{L({\bf{T}})/k({\bf{T}})}$ are $H$, ${\rm{Aut}}(E/k({\bf{T}}))$, and ${\rm{Aut}}(L({\bf{T}})/k({\bf{T}}))$, respectively, $H$ is contained in the domain of ${\rm{res}}_{L/k}^{\widehat{E}/k({\bf{T}})}$. Combine this, \eqref{eq1.1}, and \eqref{eq1.5} to get that $H_{\bf{t}}$ is contained in the domain of ${\rm{res}}_{L/k}^{\widehat{E}_{\bf{t}}/k}$. Hence, as ${\rm{res}}^{\widehat{E}_{\bf{t}}/k}_{E_{\bf{t}}/k}$ is surjective, the domain of the map ${\rm{res}}^{E_{\bf{t}}/k}_{L/k}$, which is well-defined as $L \subseteq E_{\bf{t}}$, is the whole automorphism group ${\rm{Aut}}(E_{\bf{t}}/k)$ \footnote{We use here a special case of the following easy claim: if $L/k$, $F/M$, and $N/K$ are three finite separable extensions with $k \subseteq M \subseteq K$ and $L \subseteq F \subseteq N$, then one has ${\rm{res}}^{N/K}_{F/M}(H^{N/K}_{L/k} \cap H^{N/K}_{F/M}) \subseteq H_{L/k}^{F/M}$, where $H^{N/K}_{F/M}$, $H^{N/K}_{L/k}$, and $H_{L/k}^{F/M}$ denote the domains of the maps ${\rm{res}}^{N/K}_{F/M}$, ${\rm{res}}^{N/K}_{L/k}$, and ${\rm{res}}^{F/M}_{L/k}$, respectively.}. Moreover, the domain of the map
${\rm{res}}^{{E}/k({\bf{T}})}_{L/k}$ is ${\rm{Aut}}({E}/k({\bf{T}}))$ (since those of ${\rm{res}}^{{E}/k({\bf{T}})}_{L({\bf{T}})/k({\bf{T}})}$ and ${\rm{res}}^{L({\bf{T}})/k({\bf{T}})}_{L/k}$ are ${\rm{Aut}}(E/k({\bf{T}}))$ and ${\rm{Aut}}(L({\bf{T}})/k({\bf{T}}))$, respectively). Now, use \eqref{eq1} and the surjectivity of ${\rm{res}}^{\widehat{E}_{\bf{t}}/k}_{E_{\bf{t}}/k}$ and ${\rm{res}}^{\widehat{E}/k({\bf{T}})}_{{E}/k({\bf{T}})}$ to get that the isomorphism ${\psi_{\bf{t}}}|_{{H_{\bf{t}}}}$ induces an isomorphism $h_{\bf{t}} : {\rm{Aut}}(E_{\bf{t}}/k) \rightarrow {\rm{Aut}}(E/k({\bf{T}}))$ that satisfies
\begin{equation} \label{equ}
h_{\bf{t}} \circ {\rm{res}}^{\widehat{E}_{\bf{t}}/k}_{{E}_{\bf{t}}/k} = {\rm{res}}^{\widehat{E}/k({\bf{T}})}_{{E}/k({\bf{T}})} \circ {\psi_{\bf{t}}}{|_{H_{\bf{t}}}}.
\end{equation}
\vspace{-4mm}
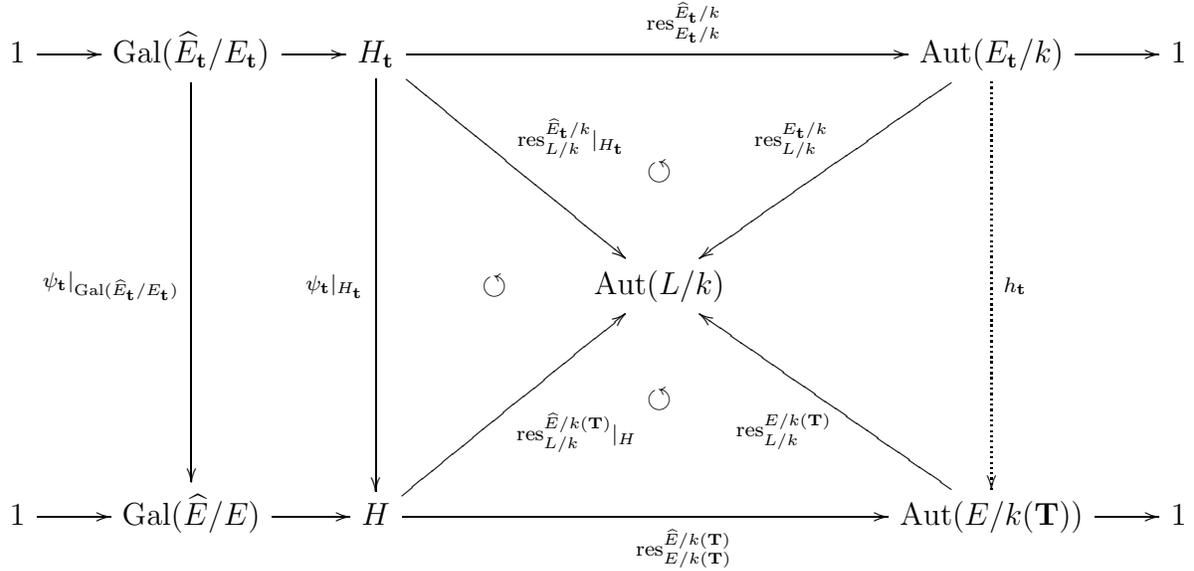
\begin{figure}[h!]
\[ \xymatrix{
1 \ar[r] & {\rm{Gal}}(\widehat{E}_{\bf{t}}/E_{\bf{t}}) \ar[r] \ar[dddd]_{\psi_{\bf{t}}|_{{\rm{Gal}}(\widehat{E}_{\bf{t}}/E_{\bf{t}})}} & H_{\bf{t}} \ar[rrdd]^{{{{{\rm{res}}^{\widehat{E}_{\bf{t}}/k}_{L/k}}}}{| _{H_{\bf{t}}}}} \ar[dddd]_{{\psi_{\bf{t}}}{|_{H_{\bf{t}}}}} \ar[rrrr]^{{\rm{res}}^{\widehat{E}_{\bf{t}}/k}_{{E}_{\bf{t}}/k}} & & & & {\rm{Aut}}(E_{\bf{t}}/k) \ar[ldld]_{{\rm{res}}^{{E}_{\bf{t}}/k}_{L/k}} \ar@{.>}[dddd]^{h_{\bf{t}}} \ar[r] & 1 \\
& & & & \circlearrowleft & & & \\
& & & \circlearrowleft & {\rm{Aut}}(L/k) & & & \\
& & & & \circlearrowleft & & & \\
1 \ar[r] & {\rm{Gal}}(\widehat{E}/E) \ar[r]& H \ar[uurr]_{{{{{\rm{res}}^{\widehat{E}/k({\bf{T}})}_{L/k}}}}{|_{H}}} \ar[rrrr]_{{\rm{res}}^{\widehat{E}/k({\bf{T}})}_{{E}/k({\bf{T}})}} & & & & {\rm{Aut}}(E/k({\bf{T}})) \ar[lulu]^{{\rm{res}}^{{E}/k({\bf{T}})}_{L/k}} \ar[r] & 1
}
\]
\caption{Group homomorphisms} \label{Figu}
\end{figure}

\noindent
By \eqref{equ}, the surjectivity of ${\rm{res}}^{\widehat{E}_{\bf{t}}/k}_{E_{\bf{t}}/k}$, and the commutativity of the three triangles in Figure \ref{Figu} (denoted by $\circlearrowleft$), one has
\begin{equation} \label{eq1.6}
{\rm{res}}^{{E}/k({\bf{T}})}_{L/k} \circ h_{\bf{t}} = {\rm{res}}^{E_{\bf{t}}/k}_{L/k}.
\end{equation}
Finally, let $\beta : {\rm{Aut}}(E/k({\bf{T}})) \rightarrow G$ be a solution to $\alpha_{k({\bf{T}})}$ whose solution field is $E$. Consider the isomorphism
$\beta \circ h_{\bf{t}} : {\rm{Aut}}(E_{\bf{t}}/k) \rightarrow G.$
By \eqref{eq1.6} and as
$\alpha \circ \beta= {{\rm{res}}_{L/k}^{E/k({\bf{T}})}},$ one has
$$\alpha \circ \beta \circ h_{\bf{t}} = {\rm{res}}^{E_{\bf{t}}/k}_{L/k},$$
as needed.

Secondly, we prove (1). Let $P({\bf{T}},X) \in k[{\bf{T}}][X]$ be the minimal polynomial of a primitive element of $\widehat{E}$ over $k({\bf{T}})$, assumed to be integral over $k[{\bf{T}}]$. As $k$ has been assumed to be Hilbertian, for each ${\bf{t}}$ in a Zariski-dense subset of $k^n$, the polynomial $P({\bf{t}},X)$ is irreducible over $k$ and separable. In particular, one has $[\widehat{E}_{\bf{t}} : k]=[\widehat{E}:k({\bf{T}})]$ and the extension ${\widehat{E}}_{\bf{t}}/k$ is Galois (see \S\ref{ssec:ffe}). Then (1) follows from the first part of the proof.

Thirdly, we prove (2). By the first part of the proof, it suffices to find ${\bf{t}} \in k^n$ such that $[\widehat{E}_{\bf{t}} : k]=[\widehat{E}:k({\bf{T}})]$, the extension $\widehat{E}_{\bf{t}}/k$ is Galois, and $E_{\bf{t}} \cap \overline{\kappa} = L \cap \overline{\kappa}$. Note that the existence of ${\bf{t}}$ such that the first two conditions hold follows immediately from the second part of the proof and the fact that $\kappa(T)$ is Hilbertian. To get the extra conclusion that at least one solution is regular, we need to specialize $T_1, \dots, T_n$ suitably. Set $F= \widehat{E} \cap \overline{\kappa}$ and let $P(T, {\bf{T}}, X) \in F[T, {\bf{T}}, X]$ be the minimal polynomial of a primitive element of $\widehat{E}$ over $F(T, {\bf{T}})$, assumed to be integral over $F[T, {\bf{T}}]$. Moreover, set $F'=E \cap \overline{\kappa}$ and let $Q(T, {\bf{T}}, X) \in F'[T, {\bf{T}}, X]$ be the minimal polynomial of a primitive element of ${E}$ over $F'(T, {\bf{T}})$, assumed to be integral over $F'[T, {\bf{T}}]$. Clearly, $P(T, {\bf{T}}, X)$ and $Q(T, {\bf{T}}, X)$ are irreducible over $\overline{\kappa}(T, {\bf{T}})$. Then apply either \cite[Proposition 13.2.1]{FJ08} and an induction on $n$ if $\kappa$ is infinite or \cite[Theorem 13.4.2 and Proposition 16.11.1]{FJ08} if $\kappa$ is finite to get the existence of ${\bf{t}} = (t_1, \dots, t_n) \in k^n$ such that $P(T, {\bf{t}}, X) \in F(T)[X]$ and $Q(T, {\bf{t}}, X) \in F'(T)[X]$ are irreducible over $\overline{\kappa}(T)$ and separable. Let $M$ be the field generated over $F(T)$ by one root of $P(T, {\bf{t}}, X)$. As this polynomial is irreducible over $F(T)$, one has $[M:\kappa(T)] = [\widehat{E}:\kappa(T, {\bf{T}})]$ and, by \S\ref{ssec:ffe}, the fields $M$ and $\widehat{E}_{\bf{t}}$ coincide (in particular, $\widehat{E}_{\bf{t}}/k$ is Galois). Then, by \S\ref{ssec:ffe}, the field $E_{\bf{t}}$ is well-defined. Moreover, $E_{\bf{t}}$ contains a root $x$ of $Q(T, {\bf{t}}, X)$. As this polynomial is irreducible over $F'(T)$, one has $E_{\bf{t}} = F'(T, x)$. Then combine this equality and the irreducibility of $Q(T, {\bf{t}}, X)$ over $\overline{\kappa}(T)$ to get $E_{\bf{t}} \cap \overline{\kappa} = F'= E \cap \overline{\kappa}$, thus ending the proof.
\end{proof}

\subsection{On the existence of geometric solutions after base change}

The aim of the next proposition is to provide a geometric Galois solution to any given finite Galois embedding problem over any given field, up to making a regular finitely generated base change.

Recall that a field extension $k_0/k$ is {\it{regular}} if $k_0/k$ is separable (in the sense of non-necessarily algebraic extensions; see, e.g., \cite[\S2.6]{FJ08}) and the equality $k_0 \cap \overline{k}=k$ holds.

\begin{proposition} \label{prop:sssecI}
Let $\alpha: G \rightarrow {\rm{Gal}}(L/k)$ be a finite Galois embedding problem over $k$. Then there exists a regular finitely generated extension $k_0/k$ such that the finite Galois embedding problem $\alpha_{k_0}$ has a geometric Galois solution\footnote{Since $k_0/k$ is regular, the fields $L$ and $k_0$ are linearly disjoint over $k$, thus making $\alpha_{k_0}$ well-defined.}.
\end{proposition}

The proof requires the next two results, that are more or less known to experts. The first one, which is Proposition \ref{prop:sssecI} for PAC fields, shows that one can take $k_0=k$ in this case.

Recall that $k$ is said to be {\it{Pseudo Algebraically Closed}} (PAC) if every non-empty geometrically irreducible $k$-variety has a Zariski-dense set of $k$-rational points. See, e.g., \cite{FJ08} for more on PAC fields.

\begin{proposition} \label{prop:pac}
Assume $k$ is PAC. Then every finite Galois embedding problem over $k$ has a geometric Galois solution.
\end{proposition}

\begin{proof}[Comments on proof]
It is a classical and deep result in field arithmetic that every finite Galois embedding problem over an arbitrary PAC Hilbertian field $k$ has a Galois solution; see \cite[Theorem A]{FV92}, \cite{Pop96}, and \cite[Theorem 5.10.3]{Jar11}. To our knowledge, the natural strengthening we consider in Proposition \ref{prop:pac} does not appear in the literature. Recall that the classical proof of the former consists in using the {\it{projectivity}} of the absolute Galois group of $k$ to reduce to solving some finite Galois split embedding problem over $k$ (see, e.g., \cite[Theorem 11.6.2 and Proposition 22.5.9]{FJ08} for more details), which can then be done by making use of \cite[Main Theorem A]{Pop96} and the Galois analogue of Proposition \ref{lemma:0}(1). We give in Appendix \ref{app} a similar and self-contained argument in our geometric context, with the necessary adjustments due to the higher generality.
\end{proof}

\begin{remark}
It is not true in general that if the field $k$ is PAC and $T$ denotes an indeterminate, then every finite Galois embedding problem over $k(T)$ (non-necessarily ``constant") has a (regular) Galois solution. For example, if $k$ is of characteristic zero and not algebraically closed, then the absolute Galois group of $k(T)$ is not projective, cf. \cite[Chapter II, Proposition 11]{Ser97}\footnote{On the other hand, if the field $k$ is separably closed, then the absolute Galois group of $k(T)$ is projective; see, e.g., \cite[Proposition 9.4.6]{Jar11}. In particular, by \cite[Proposition 22.5.9]{FJ08} and as Conjecture \ref{conj:DD1} holds over separably closed fields, all finite Galois embedding problems over $k(T)$ have regular Galois solutions.}.
\end{remark}

The second statement we need shows that if one can provide a geometric Galois solution to a given finite Galois embedding problem after some linearly disjoint base change, then one can further require such a base change to be finitely generated:

\begin{lemma} \label{prop:desc}
Let $\alpha: G \rightarrow {\rm{Gal}}(L/k)$ be a finite Galois embedding problem over $k$. Suppose there exists a field extension $M/k$ such that $M$ and $L$ are linearly disjoint over $k$ and such that $\alpha_M$ has a geometric Galois solution. Then there exists a finitely generated subextension $k_0/k$ of $M/k$ such that $\alpha_{k_0}$ has a geometric Galois solution\footnote{For every intermediate field $k \subseteq k_0 \subseteq M$, the fields $k_0$ and $L$ are linearly disjoint over $k$, thus making $\alpha_{k_0}$ well-defined.}.
\end{lemma}

\begin{proof}[Comments on proof]
For split embedding problems, the lemma may be proved by following the lines of Part B of the proof of \cite[Lemma 5.9.1]{Jar11}. We offer in Appendix \ref{app} a similar and self-contained argument which applies to any finite Galois embedding problem.
\end{proof}

\begin{proof}[Proof of Proposition \ref{prop:sssecI}]
Let $M$ be an arbitrary PAC field which is regular over $k$. See, e.g., \cite[Proposition 13.4.6]{FJ08} for an example of such a field. Consider the finite Galois embedding problem $\alpha_M$ (which is well-defined from the regularity condition). As $M$ is PAC, we may apply Proposition \ref{prop:pac} to get that $\alpha_M$ has a geometric Galois solution. It then remains to apply Lemma \ref{prop:desc} to finish the proof of Proposition \ref{prop:sssecI}.
\end{proof}

\begin{remark}
With the notation of Proposition \ref{prop:sssecI}, if $\alpha$ splits, then it is enough that the field $M$ we consider in the proof is ample (instead of PAC) and one may replace Proposition \ref{prop:pac} by \cite[Main Theorem A]{Pop96}. In particular, if $X$ denotes an indeterminate, taking $M$ equal to the Henselization of $k(X)$ with respect to the $X$-adic valuation yields the extra conclusion that $k_0$ can be chosen to have transcendence degree at most 1 over $k$.
\end{remark}

\section{Proof of Theorem \ref{thm:intro}} \label{sec:proof}

The aim of this section is to prove the following result, which generalizes Theorem \ref{thm:intro}:

\begin{theorem} \label{thm:main}
Let $k$ be an arbitrary field and $T$ an indeterminate. Then each finite embedding problem over $k(T)$ has a regular solution.
\end{theorem}

By combining Proposition \ref{lemma:0}(1) and Theorem \ref{thm:main}, we immediately get the following corollary, which generalizes Corollary \ref{coro:intro}:

\begin{corollary} \label{coro:spec}
Every finite embedding problem over a Hilbertian field has a solution.
\end{corollary}

We split the proof of Theorem \ref{thm:main} into two parts. First, we claim it suffices to ``regularly" solve embedding problems over arbitrary fields after adjoining finitely many indeterminates:

\begin{proposition} \label{thm:cst}
Let $k$ be a field and $\alpha : G \rightarrow {\rm{Aut}}(L/k)$ a finite embedding problem over $k$. Then there exists a finite non-empty tuple ${\bf{T}}$ of algebraically independent indeterminates such that $\alpha_{k({\bf{T}})}$ has a solution, whose associated solution field $E$ satisfies $E \cap \overline{k} = L$.
\end{proposition}

\begin{proof}[Proof of Theorem \ref{thm:main} under Proposition \ref{thm:cst}]
Let $k$ be a field, $T$ an indeterminate, and $\alpha : G \rightarrow {\rm{Aut}}(L/k(T))$ a finite embedding problem over $k(T)$. By Proposition \ref{thm:cst}, there exist a finite non-empty tuple ${\bf{T}}$ of algebraically independent indeterminates and a solution to the finite embedding problem $\alpha_{k(T,{\bf{T}})}$, whose solution field $E$ satisfies $E \cap \overline{k(T)}=L$. In particular, one has $E \cap \overline{k} = L \cap \overline{k}$. It then remains to apply Proposition \ref{lemma:0}(2) to get Theorem \ref{thm:main}.
\end{proof}

We now proceed to prove Proposition \ref{thm:cst} and start by recalling the following result, which is \cite[Proposition 2.3]{LP18}:

\begin{proposition} \label{lemma 11}
Given a field $k$ and $y \in k$, set $P_y(T,X)=X^3 + (T-y) X + (T-y) \in k[T][X].$ Then the polynomial $P_y(T,X)$ is irreducible, separable, and of Galois group $S_3$ over $k(T)$. Moreover, if $k_y$ denotes the field generated over $k(T)$ by any given root of $P_y(T,X)$, then, given $y_1 \not=y_2$ in $k$, one has $k_{y_1} \not= k_{y_2}$.
\end{proposition}

Let $k$ be a field and $\alpha : G \rightarrow {\rm{Aut}}(L/k)$ a finite embedding problem over $k$. Set $L'=L^{{\rm{Aut}}(L/k)}$. As ${\rm{Aut}}(L/k) = {\rm{Gal}}(L/L')$, the finite embedding problem $\alpha$ over $k$ can be seen as a finite Galois embedding problem $\alpha' : G \rightarrow {\rm{Gal}}(L/L')$ over $L'$. By Proposition \ref{prop:sssecI}, there exists an extension $k_0$ of $L'$ that is regular and finitely generated, and such that the finite Galois embedding problem $\alpha'_{k_0}$ has a geometric Galois solution. That is, there exist an indeterminate $Z$, a finite Galois extension $N/k_0(Z)$ such that $Lk_0(Z) \subseteq N$ and $N \cap \overline{k_0}=Lk_0$, and an isomorphism
$\beta : {\rm{Gal}}(N/k_0(Z)) \rightarrow G$
such that $\alpha' \circ \beta = {\rm{res}}^{N/k_0(Z)}_{L/L'}$, i.e.,
\begin{equation} \label{!!}
\alpha \circ \beta = {\rm{res}}^{N/k_0(Z)}_{L/k}.
\end{equation}

Let ${\bf{Y}}$ be a separating transcendence basis of $k_0$ over $L'$. Since the extensions $k_0/L'({\bf{Y}})$ and $L'({\bf{Y}})/k({\bf{Y}})$ are finite and separable, the same is true for the extension $k_0/k({\bf{Y}})$. Hence, there exists $y \in k_0$ such that $k_0=k({\bf{Y}},y).$ Let $T$ be an extra indeterminate. Then consider
$$P_{y}(T,X) = X^3+(T-y)X + (T-y) \in k_0[T][X].$$
Let $x$ be a root of $P_{y}(T,X)$ and denote the field $N(T,x)$ by $E$. By Proposition \ref{lemma 11}, one has
$$3=[k_0(T,x) : k_0(T)]=[k_0(Z,T,x) : k_0(Z,T)].$$
\vspace{-12mm}
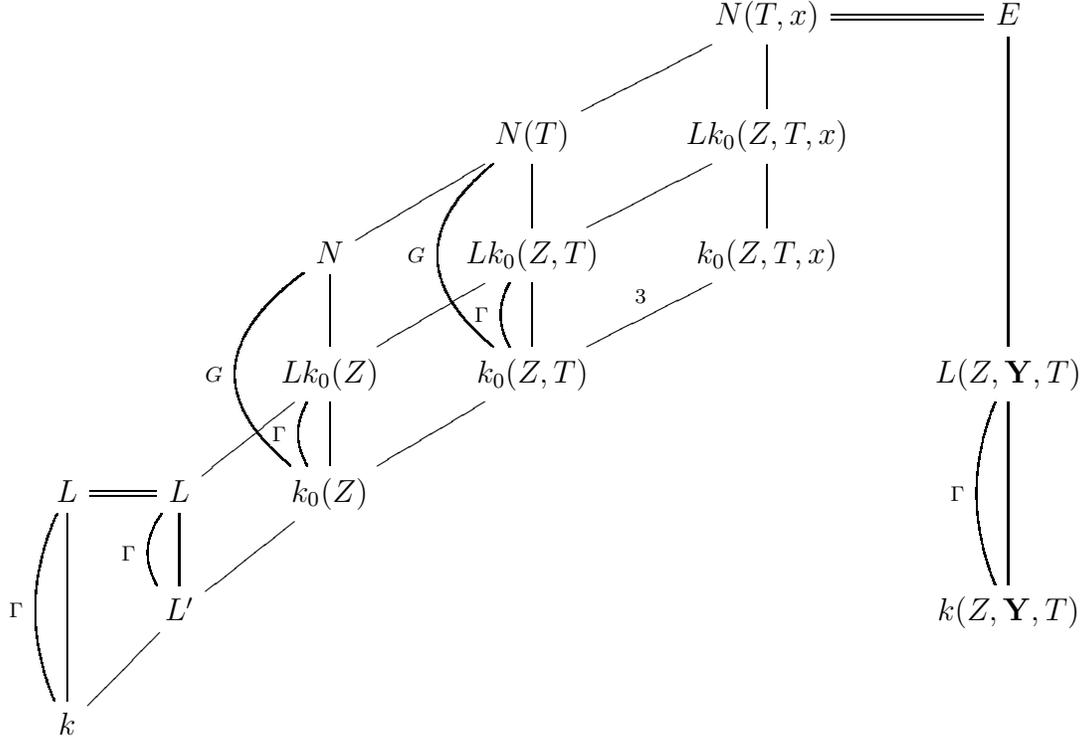
\begin{figure}[h!]
\[ \xymatrix{
& & & & N(T,x) \ar@{-}[d]  \ar@{=}[r] & E \ar@{-}[ddd] \\
& & & N(T) \ar@{-} @/_3pc/[dd]_{G} \ar@{-}[ru] \ar@{-}[d]  & Lk_0(Z,T,x) \ar@{-}[d]  & \\
& & N \ar@{-} @/_3pc/[dd]_{G} \ar@{-}[ru] \ar@{-}[d]  & Lk_0(Z,T) \ar@{-} @/_1pc/[d]_{\Gamma} \ar@{-}[ru] \ar@{-}[d]  & k_0(Z,T,x) & \\
& & Lk_0(Z) \ar@{-} @/_1pc/[d]_{\Gamma} \ar@{-}[ru] \ar@{-}[d]  & k_0(Z,T) \ar@{-}[ru]^3 & & L(Z,{\bf{Y}},T) \ar@{-}[dd] \ar@{-} @/_1pc/[dd]_{\Gamma} \\
L \ar@{-} @/_1pc/[dd]_{\Gamma} \ar@{=}[r] \ar@{-}[dd] & L \ar@{-} @/_1pc/[d]_{\Gamma} \ar@{-}[ru] \ar@{-}[d] & k_0(Z) \ar@{-}[ru] & & &  \\
& L' \ar@{-}[ru] & & & & k(Z,{\bf{Y}},T) \\
k \ar@{-}[ru] & & & & &
}
\]
\caption{Construction of the function field $E$}\label{Fig?}
\end{figure}

Now, we need the following lemma:

\begin{lemma} \label{lem:sssecIII}
One has ${\rm{Aut}}(E/k(Z,{\bf{Y}},T)) = {\rm{Gal}}(E/k_0(Z,T,x)).$
\end{lemma}

\begin{proof}
We break the proof into two parts.

Firstly, one has ${\rm{Aut}}(E/k(Z,{\bf{Y}},T)) = {\rm{Aut}}(E/k_0(Z, T)).$
Indeed, recall that $k_0(Z,T) = k({\bf{Y}}, y, Z, T)$. Let $\sigma \in {\rm{Aut}}(E/k(Z, {\bf{Y}}, T)) \setminus {\rm{Aut}}(E/k_0(Z, T))$. Then one has $\sigma(y)\not=y$ and $\sigma(x)$ is a root of
$$P_{\sigma(y)}(T,X)=X^3 + (T-\sigma(y))X + (T-\sigma(y)) \in \widehat{N}[T][X],$$
where $\widehat{N}$ denotes the Galois closure of $N$ over $k(Z, {\bf{Y}})$. Proposition \ref{lemma 11} then gives
$\widehat{N}(T, \sigma(x)) \not=\widehat{N}(T,x).$
As $\widehat{N}$ is the compositum of the $k(Z, {\bf{Y}})$-conjugates of $N$, one has
${N}(T, \sigma(x)) \not={N}(T,x).$ Since $\sigma(E) \subseteq E$, we get that ${N}(T, \sigma(x))$ is strictly contained in ${N}(T,x)$ (which is $E$). However, by Proposition \ref{lemma 11}, one has 
$$[N(T,x) : N(T)]=3= [\widehat{N}(T,\sigma(x)) : \widehat{N}(T)] \leq [{N}(T,\sigma(x)) : {N}(T)],$$ 
thus providing a contradiction.

Secondly, one has ${\rm{Aut}}(E/k_0(Z,T)) = {\rm{Gal}}(E/k_0(Z,T,x)).$ Indeed, one has to show that each $\sigma \in {\rm{Aut}}(E/k_0(Z,T))$ fixes $x$. Assume $\sigma$ does not. Then $\sigma(x)$ is another root of $P_{y}(T,X)$ and it is in $E$. Hence, $E$ contains all the roots of $P_{y}(T,X)$ (as this polynomial has degree 3 in $X$). By Proposition \ref{lemma 11}, $[E:N(T)] = 6$, a contradiction.
\end{proof}

Next, by Proposition \ref{lemma 11}, the polynomial $P_{y}(T,X)$ is irreducible over $N(T)$, that is, the map ${\rm{res}}^{E/k_0(Z,T,x)}_{N(T)/k_0(Z,T)}$ is an isomorphism. Moreover, the map ${\rm{res}}^{N(T)/k_0(Z,T)}_{N/k_0(Z)}$ is an isomorphism. Then apply Lemma \ref{lem:sssecIII} to get that the map
$$\left \{ \begin{array} {ccc}
 {\rm{Aut}}(E/k(Z,{\bf{Y}},T)) & \longrightarrow & {\rm{Gal}}(N/k_0(Z)) \\
    \sigma & \longmapsto & \sigma|_{N} \end{array} \right. $$
is a well-defined isomorphism, which we denote by ${\rm{res}}^{E/k(Z, {\bf{Y}}, T)}_{N/k_0(Z)}$ \footnote{with a slight abuse of notation as $k_0(Z) $ is not necessarily contained in $k(Z, {\bf{Y}}, T)$.}. In particular, the map
$$\beta \circ {\rm{res}}^{E/k(Z, {\bf{Y}}, T)}_{N/k_0(Z)} : {\rm{Aut}}(E/k(Z, {\bf{Y}}, T)) \rightarrow G$$
is a well-defined isomorphism. Moreover, by Lemma \ref{lem:sssecIII}, the domain of ${\rm{res}}^{E/k(Z, {\bf{Y}}, T)}_{L(Z, {\bf{Y}}, T)/k(Z, {\bf{Y}}, T)}$ is the whole automorphism group ${\rm{Aut}}(E/k(Z, {\bf{Y}}, T))$ and, by \eqref{!!}, one has
$$\alpha_{k(Z, {\bf{Y}},T)} \circ \beta \circ {\rm{res}}^{E/k(Z, {\bf{Y}}, T)}_{N/k_0(Z)} = {{\rm{res}}^{L(Z, {\bf{Y}},T)/k(Z, {\bf{Y}},T)}_{L/k}}^{-1} \circ \alpha \circ \beta \circ {\rm{res}}^{E/k(Z, {\bf{Y}}, T)}_{N/k_0(Z)} = {\rm{res}}^{E/k(Z, {\bf{Y}}, T)}_{L(Z, {\bf{Y}},T)/k(Z, {\bf{Y}},T)}.$$

Finally, $E$ is regular over $L$. Indeed, by Proposition \ref{lemma 11}, $P_y(T,X)$ is irreducible over $\overline{N}(T)$, i.e., $E$ is regular over $N$. Moreover, $N$ is regular over $Lk_0$ and $k_0$ is regular over $L'$. Then, from the latter and, e.g., \cite[Corollary 2.6.8(a)]{FJ08}, $Lk_0$ is regular over $L$. It then remains to apply, e.g., \cite[Corollary 2.6.5(a)]{FJ08} to finish the proof of Proposition \ref{thm:cst}.

\appendix

\section{Proofs of Proposition \ref{prop:pac} and Lemma \ref{prop:desc}} \label{app}

\subsection{Proof of Proposition \ref{prop:pac}} \label{app1.1}

Let $k$ be a PAC field and $\alpha : G \rightarrow {\rm{Gal}}(L/k)$ a finite Galois embedding problem over $k$. We show below that $\alpha$ has a geometric Galois solution.

Firstly, we provide a geometric Galois solution to some finite Galois split embedding problem over $k$ which ``dominates" $\alpha$. Let $k^{\rm{sep}}$ be a separable closure of $k$. As $k$ is PAC, one may apply \cite[Theorem 11.6.2]{FJ08} to get the existence of a (continuous) homomorphism $\gamma: {\rm{Gal}}(k^{\rm{sep}}/k) \rightarrow G$ that satisfies
\begin{equation} \label{D1}
\alpha \circ \gamma = {\rm{res}}^{k^{\rm{sep}}/k}_{L/k}.
\end{equation}
Let $L'$ denote the fixed field of ${\rm{ker}}(\gamma)$ in $k^{\rm{sep}}$. The extension $L'/k$ is finite and Galois, and the homomorphism $\gamma$ induces a homomorphism $\gamma' : {\rm{Gal}}(L'/k) \rightarrow G$ that satisfies
\begin{equation} \label{D2}
\gamma' \circ {\rm{res}}^{k^{\rm{sep}}/k}_{L'/k} = \gamma.
\end{equation}
By \eqref{D1} (which implies  $L \subseteq L'$) and \eqref{D2}, one has
\begin{equation} \label{D3}
\alpha \circ \gamma'= {\rm{res}}^{L'/k}_{L/k}.
\end{equation}
Let
$$G'= \{(g, \sigma) \in G \times {\rm{Gal}}(L'/k) \, \, : \, \, \alpha(g) = {\rm{res}}^{L'/k}_{L/k}(\sigma)\}$$
denote the fiber product of $G$ and ${\rm Gal}(L'/k)$ over ${\rm Gal}(L/k)$ and let $\alpha' : G' \rightarrow {\rm{Gal}}(L'/k)$ be the projection on the second coordinate. Set
$$\delta: \left \{ \begin{array} {ccc}
 {\rm{Gal}}(L'/k) & \longrightarrow & G' \\
   \sigma & \longmapsto & (\gamma'(\sigma), \sigma).
\end{array} \right.$$
By \eqref{D3}, the map $\alpha'$ is surjective and the map $\delta$ is a well-defined homomorphism. Moreover, one has $\alpha' \circ \delta = {\rm{id}}_{{\rm{Gal}}(L'/k)}.$ In particular, the finite Galois embedding problem $\alpha'$ over $k$ splits. Now, since the field $k$ is PAC, it is ample. One may then apply \cite[Main Theorem A]{Pop96} to get the existence of an indeterminate $T$, a finite Galois extension $E/k(T)$ such that $L'(T) \subseteq E$ and $E \cap \overline{k}=L'$, and an isomorphism $\beta' : {\rm{Gal}}(E/k(T)) \rightarrow G'$ such that the following equality holds:
\begin{equation} \label{D4}
\alpha' \circ \beta' = {\rm{res}}^{E/k(T)}_{L'/k}.
\end{equation}

Secondly, we deduce a Galois solution to $\alpha_{k(T)}$. Let $\beta : G' \rightarrow G$ be the projection on the first coordinate. Clearly, one has $\alpha \circ \beta = {\rm{res}}^{L'/k}_{L/k} \circ \alpha'.$
Then, by \eqref{D4}, one has
\begin{equation} \label{D6}
\alpha \circ \beta \circ \beta' = {\rm{res}}^{E/k(T)}_{L/k}.
\end{equation}
Let $E'$ be the fixed field of ${\rm{ker}}(\beta \circ \beta')$ in $E$. The epimorphism $\beta \circ \beta'  : {\rm{Gal}}(E/k(T)) \rightarrow G$ induces an isomorphism $\epsilon : {\rm{Gal}}(E'/k(T)) \rightarrow G$ that satisfies
\begin{equation} \label{D7}
\epsilon \circ {\rm{res}}^{E/k(T)}_{E'/k(T)} = \beta \circ \beta'.
\end{equation}
By the definition of $\beta$, one has
\begin{equation} \label{D7.5}
{\rm{ker}}(\beta) = \{1_G\} \times {\rm{Gal}}(L'/L).
\end{equation}
Then combine \eqref{D4} and \eqref{D7.5} to get $L(T) \subseteq E'$. It then remains to combine \eqref{D6}, \eqref{D7}, and the latter inclusion to get
$$\alpha_{k(T)} \circ \epsilon = {\rm{res}}^{E'/k(T)}_{L(T)/k(T)}.$$

Thirdly, we show that $E'/L$ is regular. Since $E'=E^{{\rm ker}(\beta\circ\beta')}$ and $E \cap \overline{k} = L'$, one has
$$L \subseteq E' \cap \overline{k} \subseteq E' \cap E \cap \overline{k} = E' \cap L' \subseteq L'^{{\rm{res}}^{E/k(T)}_{L'/k}({\rm ker}(\beta\circ\beta'))}.$$
Then use successively \eqref{D4} and \eqref{D7.5} to get
$${\rm{res}}^{E/k(T)}_{L'/k}({\rm ker}(\beta\circ\beta')) = \alpha'({\rm{ker}}(\beta)) = {\rm{Gal}}(L'/L).$$
Hence, one has $L \subseteq  E' \cap \overline{k} \subseteq L'^{{\rm{Gal}}(L'/L)}=L$, thus ending the proof of Proposition \ref{prop:pac}.

\subsection{Proof of Lemma \ref{prop:desc}} \label{app1.2}

By our assumption, there exist an extension $M/k$ such that the fields $M$ and $L$ are linearly disjoint over $k$, an indeterminate $T$, a finite Galois extension $E/M(T)$ such that $LM(T) \subseteq E$ and $E \cap \overline{M} = LM$, and an isomorphism $\beta : {\rm{Gal}}(E/M(T)) \rightarrow G$ such that
\begin{equation} \label{!}
\alpha \circ \beta = {\rm{res}}^{E/M(T)}_{L/k}.
\end{equation}

Let $x \in E$ be such that $E=M(T,x)$. Let $P(X) \in M(T)[X]$ be the minimal polynomial of $x$ over $M(T)$. Let 
$$x_2, \dots, x_{|G|}$$ 
be the roots of $P(X)$ that are not equal to $x$. As the extension $E/M(T)$ is Galois, for each $i \in \{2, \dots, |G|\}$, there is a polynomial $P_i(X) \in M(T)[X]$ such that $x_i=P_i(x).$ Also, pick $z \in L$ such that $L=k(z)$. Then there is $Q(X) \in M(T)[X]$ such that $z=Q(x).$ Let $k_0$ be a subfield of $M$ that is finitely generated over $k$ and such that
$$P(X), P_2(X), \dots, P_{|G|}(X), Q(X) \in k_0(T)[X].$$
Then the extension $k_0(T,x)/k_0(T)$ is Galois, one has $L \subseteq k_0(T, x)$, and the restriction map $${\rm{res}}^{E/M(T)}_{k_0(T,x)/k_0(T)}$$ 
is an isomorphism. Moreover, as the fields $M$ and $L$ are linearly disjoint over $k$ and $k \subseteq k_0 \subseteq M$, the fields $k_0$ and $L$ are linearly disjoint over $k$. That is, ${\rm{res}}^{Lk_0/k_0}_{L/k}$ is an isomorphism.
\begin{figure}[h!]
\[ \xymatrix{
& & E = M(T,x)  \ar@{-}[d]\ar@{-}[rd] \ar@{-} @/_3pc/[dd]_{G} & & & \\
& & LM(T) \ar@{-} @/_1pc/[d]_{\Gamma} \ar@{-}[d] \ar@{-}[rd] & k_0(T,x) \ar@{-} @/_3pc/[dd]_{G} \ar@{-}[d]  & &\\
& LM \ar@{-} @/_0.5pc/[d]_{\Gamma} \ar@{-}[ru] \ar@{-}[d] &M(T) \ar@{-}[rd] & Lk_0(T) \ar@{-} @/_1pc/[d]_{\Gamma} \ar@{-}[rd] \ar@{-}[d] && \\
L \ar@{-} @/_1pc/[d]_{\Gamma} \ar@{-}[d] \ar@{-}[ru] &M \ar@{-}[ru] & &k_0(T) \ar@{-}[rd] & Lk_0 \ar@{-} @/_0.5pc/[d]_{\Gamma} \ar@{-}[d] \ar@{-}[rd]&\\
k \ar@{-}[ru]& & &  &k _0 \ar@{-}[rd] & L \ar@{-} @/_0.5pc/[d]_{\Gamma} \ar@{-}[d]\\
& & & & &k}
\]
\caption{Defining the field $E$ over a suitable extension $k_0$ of $k$}\label{Fig6b}
\end{figure}
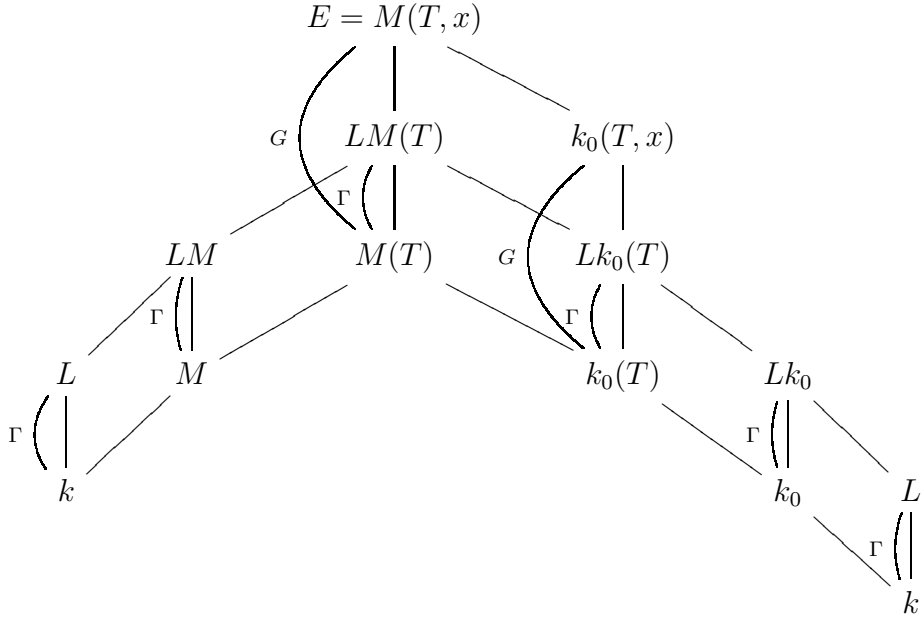

Then consider the finite Galois embedding problem $\alpha_{k_0(T)}$. The map
$$\beta \circ {{\rm{res}}^{E/M(T)}_{k_0(T,x)/k_0(T)}}^{-1} : {\rm{Gal}}(k_0(T,x)/k_0(T)) \rightarrow G$$
is an isomorphism and one has $Lk_0(T) \subseteq k_0(T,x)$. Furthermore, by \eqref{!} and since 
$$\alpha_{k_0(T)} = {{\rm{res}}^{Lk_0(T)/k_0(T)}_{Lk_0/k_0}}^{-1} \circ {{\rm{res}}^{Lk_0/k_0}_{L/k}}^{-1} \circ \alpha,$$ 
one has
$\alpha_{k_0(T)} \circ \beta \circ {{\rm{res}}^{E/M(T)}_{k_0(T,x)/k_0(T)}}^{-1}={\rm{res}}^{k_0(T,x)/k_0(T)}_{Lk_0(T)/k_0(T)}.$

Finally, $k_0(T, x) \cap \overline{k_0} = Lk_0$. Indeed, combine
$k_0(T,x) \cap \overline{k_0} \subseteq  E \cap \overline{M} = LM \subseteq LM(T)$
and the linear disjointness of the fields $LM(T)$ and $k_0(T,x)$ over $Lk_0(T)$ to get 
$$k_0(T,x) \cap \overline{k_0} \subseteq k_0(T,x) \cap LM(T) \cap \overline{k_0}= Lk_0(T) \cap \overline{k_0} = Lk_0,$$
thus ending the proof of Lemma \ref{prop:desc}.

\end{document}